\documentclass[psamsfonts]{amsart}
\usepackage[utf8]{inputenc}
\usepackage{amsfonts}
\usepackage{hyperref}
\usepackage{amsmath}
\usepackage{amssymb}
\usepackage{xcolor}
\usepackage{amsthm}
\usepackage[scr=rsfs]{mathalpha}
\usepackage[all]{xy}
\usepackage{pdflscape}
\usepackage{pgfplots}

\newtheorem{thm}{Theorem}[section]
\newtheorem{cor}[thm]{Corollary}
\newtheorem{prop}[thm]{Proposition}
\newtheorem{lem}[thm]{Lemma}

\theoremstyle{definition}
\newtheorem{defn}[thm]{Definition}

\newtheorem{exmp}[thm]{Example}

\theoremstyle{remark}
\newtheorem{rem}[thm]{Remark}

\makeatletter
\let\c@equation\c@thm
\makeatother
\numberwithin{equation}{section}

\title{Smooth Euler-symmetric varieties generated by a single polynomial}
 \author[Cong~Ding]{Cong~Ding} 
 \address[Cong~Ding]
 {School of Mathematical Sciences\\ Shenzhen University\\Guangdong\\ China}
 \email{congding@szu.edu.cn}
 
 \author[Zhijun~Luo]{Zhijun~Luo}
 \address[Zhijun~Luo]
 {Center for complex geometry\\ Institute for Basic Science\\ 55 Expo-ro\\ Yuseong-gu\\ Daejeon\\ 34126\\ Republic of Korea.}
\email{luozj@amss.ac.cn \\ luozj@ibs.re.kr}

\begin{document}

\begin{abstract}
We classify smooth Euler-symmetric varieties corresponding to the symbol system generated by a single reduced polynomial.
\end{abstract}

\maketitle
	\section{Introduction}\label{intro}
	The notion of Euler-symmetric varieties was first introduced by Hwang and Fu in \cite{Fu2017EulersymmetricPV}. They are nondegenerate projective varieties admitting many $\mathbb{C}^*$-actions of Euler type. More precisely we have
	\begin{defn}
		Let $Z \subset \mathbb{P}V$ be a nondegenerate projective variety. For a nonsingular point $x \in Z$, a $\mathbb{C}^{*}$-action on $Z$ coming from a multiplicative subgroup of $\operatorname{GL}(V)$ is said to be of Euler type at $x$ if 
		\begin{enumerate}
			\item $x$ is an isolated fixed point of the restricted $\mathbb{C}^{*}$-action on $Z$;
			\item the induced $\mathbb{C}^{*}$-action on the tangent space $T_x(Z)$ is by scalar multiplication.
		\end{enumerate}
		We say that $Z \subset \mathbb{P}V$ is Euler-symmetric if for a general point $x \in Z$, there exists a $\mathbb{C}^{*}$-action on $Z$ of Euler type at $x$.
	\end{defn}
	There are many examples of Euler-symmetric varieties. For instance, Hermitian symmetric spaces under the equivariant projective embedding with respect to their automorphism groups, blow-up of $\mathbb{P}^n$ along a subvariety inside a hyperplane (where the projective embeddings have not been classified) are Euler-symmetric. More examples can be found in \cite{Fu2017EulersymmetricPV}. 
	Characterizing or classifying Euler-symmetric varieties under specific conditions presents a fascinating challenge.
	In \cite{Fu2017EulersymmetricPV}, the authors propose a conjecture that a Fano manifold with Picard number 1, realizable as an equivariant compactification of a vector group, can also be realized as an Euler-symmetric projective variety under a suitable projective embedding. This conjecture has been solved in \cite{Shafarevich23} under the assumption that the Fano manifold is toric. In \cite{Luo23}, the second author demonstrates that any Euler-symmetric complete intersection is necessarily a complete intersection of hyperquadrics.

	From \cite{Fu2017EulersymmetricPV}, Euler-symmetric varieties are determined by the symbol systems. We first recall
	\begin{defn}
		Let $W$ be a vector space. For $w\in W$, define
		$$\iota_w: \operatorname{Sym}^{k+1}W^* \to \operatorname{Sym}^{k}W^*, \, \varphi \mapsto \varphi(w, \cdot, \ldots, \cdot).$$
		For a subspace $F^k \subset \operatorname{Sym}^{k}W^*$ of symmetric $k$-linear forms on $W$, define its prolongation $\textbf{prolong}(F^k) \subset \operatorname{Sym}^{k+1}W^*$ by the following
		$$ \textbf{prolong}(F^k):= \bigcap_{w\in W}\iota_w^{-1}(F^k).$$
	\end{defn}
	Then a symbol system is defined as follows.
	\begin{defn}
		Let $W$ be a vector space. Fix a natural number $r$. A subspace
		$$
		\textbf{F} = \oplus_{k\geq 0}F^k \subset \operatorname{Sym}(W^*):= \oplus_{k\geq 0}\operatorname{Sym}^{k}W^*
		$$
		with
		$$
		F^0 = \mathbb{C} = \operatorname{Sym}^{0}W^*, \; F^1 = W^*,\; F^r \neq 0, \; \text{and}\; F^{r+i} = 0 \; \forall \; i \geq 1,
		$$
		is called a symbol system of rank $r$, if $F^{k+1} \subset \textbf{prolong}(F^k)$ for each $1\leq k \leq r$.
	\end{defn}
	One typical example of a symbol system comes from the system of fundamental forms of a non-degenerate projective variety. Moreover we know
	\begin{thm}[Classical result due to E. Cartan]\label{cartan}
		Let $Z \subset PV$ be a nondegenerate subvariety and let $x \in Z$ be \textbf{a general point}. Then the system of fundamental forms $\textbf{F}_x = \oplus_{k\geq 0}F_x^k$ is a symbol system of rank $r$ for some natural number $r\geq 1$.
	\end{thm}
	
	\begin{defn}
		Given a symbol system $\textbf{F}$ of rank $r$, define a rational map
		$$
		\Phi_{\textbf{F}} : \mathbb{P}(\mathbb{C}\oplus W) \dashrightarrow \mathbb{P}(\mathbb{C}\oplus W \oplus (F^2)^*\oplus \cdots \oplus (F^r)^*),
		$$
		by
		$$
		[t:w] \mapsto [t^r: t^{r-1}w: t^{r-2}\iota^2_w:\cdots: t\, \iota^{r-1}_w:\iota^r_w].
		$$
		Write $V_{\textbf{F}}:=\mathbb{C} \oplus W \oplus (F^2)^*\oplus \cdots \oplus (F^r)^*$. We will denote the closure of the image of the rational map $\Phi_{\textbf{F}}$ by $M(\textbf{F}) \subset \mathbb{P}V_{\textbf{F}}$.
		We say the projective variety $M(\textbf{F})$ associated to the symbol system $\textbf{F}$ has rank $r$, denoted by $\operatorname{rank}(M(\textbf{F}))$. 
	\end{defn}
	
	The following theorem gives the relation between Euler-symmetric varieties and symbol systems.
	\begin{thm}\cite[Theorem 3.3]{Fu2017EulersymmetricPV}\label{fu20}
		Let $o= [1:0:\cdots:0]\in M(\textbf{F})$ be the point $\Phi_{\textbf{F}}([t=1: w=0])$. Then:
		\begin{enumerate}
			\item The natural action of the vector group $W$ on $\mathbb{P}(\mathbb{C}\oplus W)$ can be extended to an action of $W$ on $\mathbb{P}V_{\textbf{F}}$ preserving $M(\textbf{F})$ such that the orbit of $o$ is an open subset biregular to $W$.
			\item The $\mathbb{C}^{*}$-action on $W$ with weight $1$ induces a $\mathbb{C}^{*}$-action on $M(\textbf{F})$ of Euler type at $o$, making $M(\textbf{F})$ Euler-symmetric.
			\item The system of fundamental forms of $M(\textbf{F})\subset \mathbb{P}V_{\textbf{F}}$ at $o$ is isomorphic to the symbol system $\textbf{F}$.
		\end{enumerate}
		Conversely, any Euler-symmetric projective variety is of the form $M(\textbf{F})$ for some symbol system $\textbf{F}$ on a vector space $W$.
	\end{thm}
	\begin{rem}\label{explicitaction}
		The action of $W$ on $\mathbb{P}V_{\mathbf{F}}$ coming from
		the natural action of the vector group $W$ on $\mathbb{P}(\mathbb{C}\oplus W)$ can be explicitly written as follows. For $v \in W$ and $x = [t:w:f^2:\cdots:f^r]$,
		\[
		g_v\cdot x = [t:w+tv:g_v^x\cdot f^2:\cdots:g_v^x\cdot f^r],
		\]
		where for each $2\leq k \leq r$,
		\[
		g_v^x\cdot f^k = \sum_{\ell = 2}^k C_k^{\ell} f^{\ell}\circ \iota_v^{k-\ell} + k\iota_w\circ \iota_v^{k-1} + t\iota_v^k,
		\]
		denoting the binomial coefficients by $C^\ell_k$.
		
		The action of $\mathbb{C}^{*}$ on $\mathbb{P}V_{\mathbf{F}}$ that induced by the $\mathbb{C}^{*}$-action on $W$ with weight $1$ can be explicitly written as follows. 
		For $\lambda \in \mathbb{C}^*$ and $x = [t:w:f^2:\cdots:f^r]$,
		\[
		\lambda\cdot x = [t:\lambda w:\lambda^2 f^2:\cdots:\lambda^r f^r].
		\]
	\end{rem}
	\begin{defn}
		Let $U$ be the open orbit biregular to $W$, and the boundary divisor $D = M(\textbf{F})\backslash U$.
	\end{defn}
	
	In this article we are going to characterize certain special Euler-symmetric varieties associated to symbol systems generated by a single reduced homogeneous polynomial. 
	\begin{defn}
		Given a vector space $W$. Let $P \in \operatorname{Sym}^rW^*$ be a homogeneous polynomial of degree $r$. The associated symbol system $\textbf{F}_P := \oplus_{i=0}^r F^i_P$ is defined as follows:
		\[
		F^2_P = \langle d^{r-2}P\rangle, \cdots, F^k_P = \langle d^{r-k}P\rangle, \cdots, F^r_P = \langle P\rangle.
		\]
		In this case, we denote $V_P := V_{\textbf{F}_P}$, and the Euler-symmetric variety associated to this symbol system is denoted by $M_P\subset \mathbb{P}V_P$, correspondingly the rational map is denoted by $\Phi_P$.
	\end{defn}
	
	Our main result is the following. 
	\begin{thm}\label{mainthm}
		Given a vector space $W$. Let $P \in \operatorname{Sym}^rW^*$ be a reduced homogeneous polynomial of degree $r$. If the Euler-symmetric variety $M_P$ is smooth, then $M_P$ is a homogeneous projective variety.
		
		Moreover, $M_P = X_1\times\cdots\times X_k$ where $k$ is the number of irreducible components of $P$, and each $X_i$ is one of the following irreducible Hermitian symmetric spaces, and the embedding $M_P \subset \mathbb{P}V_P$ is given by the line bundle $\mathcal{O}(1,\cdots, 1)$.
		\begin{enumerate}
			\item the Grassmannian variety $Gr(n,2n)$;
			\item the Spinor variety $\mathbb{S}_n$ with $n$ being even;
			\item the Lagrangian Grassmannian $\operatorname{LG}(n,2n)$;
			\item the hyperquadric $\mathbb{Q}^n$;
			\item the $27$-dimensional $E_7/P_7$.
		\end{enumerate}
		Conversely, for any $X = X_1\times\cdots\times X_k \subset\mathbb{P}^N$ such that $X_i$ is one of the above irreducible Hermitian symmetric spaces with the embedding in $\mathbb{P}^N$ given by the line bundle $\mathcal{O}(1,\cdots, 1)$, there exists a reduced homogeneous polynomial $P$ such that $(X\subset \mathbb{P}^N) \cong (M_P\subset \mathbb{P}V_P)$, up to projective isomorphism.
	\end{thm}
	
	The key point of the proof lies in demonstrating that the system of fundamental forms at the terminal point $z=[0,\cdots, 0, 1]$ constitutes a symbol system. 
	This presents a challenge because Cartan's Theorem, which holds on general points, does not readily provide this information. 
	Algebraically, Proposition \ref{derivative} allows us to easily obtain only the information of the $(r-1)$-th fundamental form. This form is precisely captured by the derivatives of $P_*$, the so-called multiplicative Legendre transform of $P$ (Definition \ref{multilegen}).
	However, it is not easy to check that the system of fundamental forms of $M_P$ at $z$ is exactly coincident with the symbol system generated by $P_*$. Instead of this tedious algebraic verification, we use a theorem recently proved by Lawrence Ein and Wenbo Niu (see \cite[Theorem 3.3]{ein2023vanishing}). This theorem establishes that the rank equality between the system of fundamental forms and the symbol system generated by $P$ (see Proposition \ref{symmetricFund}) will imply that the system of fundamental forms at $z$ is a symbol system. Consequently, the existence of another vector group action on $M_P$ is guaranteed, with its orbit through $z$ being biregular to the vector group.
	
	The organization of the rest of this paper is as follows: In \S2 we discuss the homaloidal polynomial and its relation with the smoothness of the Euler symmetric variety $M_P$, which is sufficient to obtain the rank equality. In \S3 we discuss some properties on the fixed point components of $M_P$ in the Bia{\l}ynicki-Birula decomposition and the boundary divisor. In \S4 we finish the proof of our main result (Theorem \ref{mainthm}).
	\section{EKP-Homaloidal polynomials and smoothness of Euler-symmetric varieties}
	We first recall 
	\begin{defn}
		Let $P \in \operatorname{Sym}^rW^*$ be a homogeneous polynomial. We say $P$ is a homaloidal polynomial, if the induced derivative map
		\[
		P' : \mathbb{P}W \to \mathbb{P}W^*, w \mapsto (d P)_w,
		\]
		is birational map.
	\end{defn}
	
	We list some cubic homaloidal polynomials here.
	\begin{exmp}\label{exam:cubichomaloidal}
		\begin{enumerate}
			\item $W = \mathbb{C}$, $P = x^3$;
			\item $W = \mathbb{C}^m\oplus \mathbb{C}$, $P = Q(v)x$, where $Q$ is any non-degenerate quadratic form on $\mathbb{C}^m$;
			\item $W = \mathbb{C}^3$, $P = x_0x_1x_2$;
			\item $W$ is the space of $3$ by $3$ symmetric matrices, $P$ is the determinant function;
			\item $W = \operatorname{Mat}_3(\mathbb{C})$, $P$ is the determinant function;
			\item $W = \Lambda^2(\mathbb{C}^6)$, $P$ is the Pffafian polynomial;
			\item $W = \operatorname{Mat}_3(\mathbb{C}) \times \operatorname{Mat}_3(\mathbb{C}) \times \operatorname{Mat}_3(\mathbb{C})$, the polynomial $P$ is the Cartan cubic $P(A, B, C) = |A| + |B| + |C| - \operatorname{Tr}(ABC)$.
		\end{enumerate}
		Moreover, the last four examples correspond to irreducible Hermitian symmetric spaces listed in Theorem \ref{mainthm} with $\operatorname{rank}=3$.
	\end{exmp}
	
	\begin{prop}\label{symmetrichomo}
		Let $P$ be a homogeneous polynomial, then
		\[
		\operatorname{dim}(F^j_P) = \operatorname{dim}(F^{r-j}_P), \forall j\neq 1.
		\]
		Moreover, if $P$ is homaloidal, then
		\[
		\operatorname{dim}(F^j_P) = \operatorname{dim}(F^{r-j}_P), \forall j.
		\]
	\end{prop}
	\begin{proof}
		Let $\operatorname{Sym}^dW$ be the vector space of degree $d$ homogeneous differential operator with constant coefficients. Consider the follow $j$-th partial derivative map of $P$:
		\[
		P_{j,r-j}:\operatorname{Sym}^jW\to \operatorname{Sym}^{r-j}W^*; \; D \mapsto D(P).
		\]
		By the definition of the symbol system $\textbf{F}_P$, we have that $\operatorname{Im}(P_{j, r-j}) = F_P^{r-j}$ for $j\neq r-1$. 
		We claim that 
		\[
		\operatorname{rank}(P_{j, r-j}) = \operatorname{rank}(P_{r-j,j}).
		\]
		This was mentioned in \cite{gesmundo2019explicit} without explicit proof.
		To prove the claim, we write 
		\[P=\sum_{\substack{I=(a_1,\cdots, a_m), a_i\geq 0 \\ a_1+\cdots +a_m=r}}A_I\dfrac{1}{a_1!\cdots a_m!}x^{a_1}_1\cdots x^{a_m}_m.\]
		For $D^j\in \operatorname{Sym}^jW$, it can be written as 
		\[D^j=\sum_{\substack{J=(b_1,\cdots, b_m), b_i\geq 0 \\ b_1+\cdots +b_m=j}}D^j_J\dfrac{\partial^j}{\partial x^{b_1}_1\cdots \partial x^{b_m}_m}.\]
		Then we have
		\[D^j(P)=\sum_{I,J, a_i\geq b_i}A_ID^j_J\dfrac{1}{(a_1-b_1)!\cdots (a_m-b_m)!}x^{a_1-b_1}_1\cdots x^{a_m-b_m}_m.\]
		Define $\Phi: \operatorname{Im}(P_{j, r-j})\rightarrow \operatorname{Im}(P_{r-j, j}) $ as
		\[\Phi(D^j(P))=\sum_{I,J, a_i\geq b_i}A_ID^j_J\dfrac{1}{(b_1)!\cdots (b_m)!}x^{b_1}_1\cdots x^{b_m}_m.\]
		It is easy to see that $\Phi$ is a bijection and hence the claim holds. Therefore, $\operatorname{dim}(F^j_P) = \operatorname{dim}(F^{r-j}_P)$ for any $j\neq 1$.
		
		If $P$ is homaloidal, then $\operatorname{dim}(F^{r-1}_P) = m = \operatorname{dim}(W^*)$.
	\end{proof}
	We fix some conventions here.
	Consider the $W$-action on $M_{\mathbf{F}}$ as stated in Theorem \ref{fu20}. For a vector $v\in W$, denoting the corresponding group element by $g_v$, the limiting point
	\[
	\lim_{t\to 0} g_{t^{-1}v}\cdot o \in M_{\mathbf{F}}
	\]
	is called the terminal point of the rational curve $\overline{\{g_{tv}\cdot o \mid t\in \mathbb{C} \}}\subset M_P$, and $o$ is called the original point. If $v$ is a general vector, the limiting points are called the terminal points of the $W$-action.
	
	Homaloidal polynomials are closely related to the smoothness of $M_P$. More precisely we have
	\begin{prop}\label{smhomal}
		If the Euler-symmetric variety $M_P$ is smooth, then $P$ is a homaloidal polynomial.
	\end{prop}
	\begin{proof}
		Consider the terminal point $z = [0:\cdots:0:1]\in M_P\subset \mathbb{P}V_{P}$. Firstly, for any point $w \in W$ such that $P(w)\neq 0$, we have that the projective tangent cone at $z$ contains the direction $[0:\cdots:0:d P(w):0]$. 
		We claim that the the map $w\mapsto d P(w)$ is not degenerate, namely the image is not contained in a hyperplane. 
		If not, by choosing a suitable coordinate, we can assume that $\partial P/\partial w_m \equiv 0$. Hence $M_P$ is degenerate. It contradicts with the non-degeneracy of Euler-symmetric projective varieties.
		Therefore, the projective tangent cone must contain the projective subspace $L = \{[0:\cdots:0:*:*]\}$ of dimension $m$. 
		
		Consider the tangential projection of $M_P$ from $z$:
		\begin{align*}
			\psi : M_P &\dashrightarrow L \\
			[t^r:t^{r-1}w:\cdots:td P(w): P(w)]& \mapsto [td P(w): P(w)].
		\end{align*}
		Since the Euler-symmetric variety $M_P$ is smooth, the projective subspace $L$ coincides with projective tangent space of $M_P$ at point $z$. Then $\psi$ must be a local isomorphism in the complex topology. That is, the map $\psi$ must be a local isomorphism at $z$, in particular it must be locally injective.
		
		On the other hand, suppose that we can find two non-colinear vectors $w_1$ and $w_2$ with $P(w_1),P(w_2)\neq 0$ such that the vectors $d P(w_1)$ and $d P(w_2)$ are colinear. 
		After multiplying them by a suitable constant, we may suppose that $P(w_1)^{-1}d P(w_1) = P(w_2)^{-1}d P(w_2)$. 
		Then for $t$ small enough, $[t^r: t^{r-1}w_1: \cdots: td P(w_1):P(w_1)]$ and $[t^r: t^{r-1}w_2: \cdots: td P(w_2):P(w_2)]$ are two distinct points in $M_P$ close to $z$.
		But $[td P(w_1):P(w_1)]$ and $[td P(w_2):P(w_2)]$ coincide, a contradiction with the local injectivity of $\psi$. 
		We conclude that the rational map $[w]\in \mathbb{P}W \mapsto [d P(w)]\in \mathbb{P}W^*$ must be injective on the open subset $P(w) \neq 0$. Hence, the polynomial $P$ is homaloidal.
	\end{proof}
	
	Since $M_P$ is smooth, 
	we consider the system of fundamental forms at $z$, which is denoted by $\textbf{F}_z = \oplus_{k=0}^{r_z}F^k_z$. Note that we can not use Theorem \ref{cartan} to deduce that $\textbf{F}_z$ is a symbol system directly since $z$ might not be a general point. However, the dimensions match as follows. 
	\begin{prop}\label{symmetricFund}
		$r_z=r$ and 
		$\operatorname{dim}(F_z^k) = \operatorname{dim}(F_P^{r-k}) = \operatorname{dim}(F_P^k)$, for any $ 0\leq k\leq r$.
	\end{prop}
	\begin{proof}
		Let $\mathscr{L}_w$ be the image of the line $L_w = \{[a:bw]\mid [a:b]\in \mathbb{P}^1, P(w)\neq 0 \}$ under the rational map $\Phi_{P}$.
		The tangent direction of $\mathscr{L}_w$ at $z$ is 
		\[
		[0:\cdots:\frac{d P(w)}{P(w)}:0],
		\]
		and the coordinate of the neighborhood of $z$ is
		\[
		[\frac{t^r}{P(w)}:\frac{t^{r-1}w}{P(w)}:\frac{t^{r-2}\iota^2_w}{P(w)}: \cdots:\frac{td P(w)}{P(w)}:1].
		\]
		Therefore, the tangent space $T_zM_P$ can be identified with $W^*$.
		Hence by \cite[Lemma 2.5]{Fu2017EulersymmetricPV} and the smoothness of $M_P$, we can take the coordinate 
		\[
		(z_{m_r}^{(r)},\cdots, z_{1}^{(r)},\cdots,z_{m_2}^{(2)},\cdots, z_{1}^{(2)},z_m,\cdots,z_1),
		\]
		such that $z_j = \frac{t\partial_j P(w) }{P(w)}$, and analytic functions $h^k_j(z_m, \cdots, z_1)$ satisfy
		\[
		z_j^{(k)} = h^k_j(z_m, \cdots, z_1), 2\leq k\leq r, 1\leq j\leq m_i.
		\]
		By the Taylor's expansion of $h^k_j$ and comparing the degree w.r.t $t$, we know that $h^k_j$ are homogeneous polynomials of degree $k$. Hence $r_z=r$ and the system of fundamental forms of $M_P$ at $z$ is $\textbf{F}_z = \oplus_{k=0}^rF^k_z$ with $F_z^k = \langle h^k_j\mid 1\leq j\leq m_k \rangle \subset \operatorname{Sym}^kW$ for $k\geq 2$. 
		Together with Proposition \ref{symmetrichomo} and Proposition \ref{smhomal} we know $m_k = \operatorname{dim}(F_P^{r-k})=\operatorname{dim}(F_P^k)$, for any $ 0\leq k\leq r$.
	\end{proof}
	
	Actually from the smoothness of $M_P$, we have more restrictions on the polynomial $P$. Recall 
	\begin{defn}\label{multilegen}
		Let $P$ be a homogeneous function and $\operatorname{det}(\operatorname{Hess}(\operatorname{ln}(P)))$ is not identically zero. 
		In this case we can define a function $P_*$ by \[P_*(\frac{d P(w)}{P(w)}) = \frac{1}{P(w)}.\] 
		If $P$ is homogeneous of degree $d$ then so is $P_*$. 
		We will call $P_*$ the multiplicative Legendre transform of $P$.
	\end{defn}
	\begin{prop}\cite[Proposition 3.6]{etingof2002fourier}\label{inverse}
		Fix a vector space $W$. Let $P \in \operatorname{Sym}^rW^*$ be a homogeneous polynomial of degree $r$, and $\operatorname{det}(\operatorname{Hess}(\operatorname{ln}(P)))$ is not identically zero. Then $P$ is a homaloidal polynomial if and only if its multiplicative Legendre transform $P_*$ is a rational function. Moreover, in this case, 
		$$
		d \operatorname{ln} P_* = (d \operatorname{ln} P)^{-1}.
		$$
	\end{prop}
	\begin{defn}
		We say a homogeneous polynomial $P$ is an EKP-homaloidal polynomial if its multiplicative Legendre transform $P_*$ is also a homogeneous polynomial.
	\end{defn}
	From \cite[Lemma 3.5]{etingof2002fourier}, we know $P_{**}=P$. Hence $P$ is an EKP-homaloidal polynomial if and only if $P_{*}$ is an EKP-homaloidal polynomial.
	Following the proof of Proposition \ref{symmetricFund}, we immediately have
	\begin{prop}\label{ekphomal}
		If the Euler-symmetric variety $M_P$ is smooth, then $P$ is an EKP-homaloidal polynomial.
	\end{prop}
	\begin{proof}
		From the proof of Proposition \ref{symmetricFund}, we know that $\frac{1}{P(w)} = h^r_1(\frac{d P(w)}{P(w)})$, namely $P_* = h^r_1$. Hence $P$ is an EKP-homaloidal polynomial.
	\end{proof}
	
	Let $\pi : \textbf{Bl}_zM_P\to M_P$, $E$ be the exceptional divisor. In fact, the inverse map $\rho := \Phi^{-1}_{P}: M_P \dashrightarrow \mathbb{P}(\mathbb{C}\oplus W)$ is the projection map. Then the composed rational map $\overline{\pi} :=\rho\circ\pi: E \dashrightarrow \mathbb{P}(\mathbb{C}\oplus W)$ is defined by the $(r-1)$-th fundamental form of $M_P$ at $z$, namely, $h^{r-1}_j$, $1\leq j\leq m$. In particular, the image of the rational map $\overline{\pi}$ is contained in $\mathbb{P}W$. 
	\begin{prop}\label{derivative}
		If the Euler-symmetric variety $M_P$ is smooth, then $F^{r-1}_z = F^{r-1}_{P_*}$. In particular, $\overline{\pi}$ is birational map.
	\end{prop}
	\begin{proof}
		By definition of $h^{r-1}_j$, we have that $h^{r-1}_{m+1-j} (\frac{dP(w)}{P(w)}) = \frac{w_j}{P(w)}$, for $w \in W$. Let $\Pi: W^* \to W, w \mapsto (h^{r-1}_m(w), \cdots, h^{r-1}_1(w))$. Identifying $E$ with $\mathbb{P}W^*$, the rational map $\overline{\pi}$ and $\Pi$ are same under the projectivization.
		Consider the morphisms
		\begin{align*}
			\Phi_1 : &\; W \to W^*, w \mapsto dP(w), \\
			\Phi_1^* : &\; W^* \to W, w \mapsto dP_*(w).
		\end{align*}
		By Proposition \ref{inverse} and the definition of EKP-homaloidal polynomials, we know that 
		\[
		\Phi_1\circ \Phi_1^*(w) = P_*^{r-2}(w)w, \Phi_1^*\circ \Phi_1(w) = P^{r-2}(w)w.
		\]
		On the other hand since $h^{r-1}_j$ are homogeneous polynomials of degree $r-1$, we also have 
		\[
		\Pi \circ \Phi_1(w) = P^{r-2}(w)w,
		\]
		therefore $\Pi \circ \Phi_1 = \Phi_1^*\circ \Phi_1$. Composed on the right by the morphsim $\Phi_1^*$, we have 
		\begin{align*}
			\Phi_1^* \circ \Phi_1 \circ \Phi_1^* (w) &= P^{(r-2)(r-1)}_*(w)\Phi_1^*(w),\\
			\Pi \circ \Phi_1 \circ \Phi_1^* (w) &= P^{(r-2)(r-1)}_*(w)\Pi(w).
		\end{align*}
		Therefore, we know that $\Pi = \Phi_1^*$.
	\end{proof}
	\section{Bia{\l}ynicki-Birula decomposition and the boundary divisor}
	
	\begin{lem}\label{multiAction}
		Let $\mathbb{G}_m \subset GL(V_{P})$ be the multiplicative subgroup corresponding to the $\mathbb{C}^*$-action on $M_P$ in Remark \ref{explicitaction}. 
		Let $\mathbb{P}V_{F_P^i} \subset \mathbb{P}V_{P}$ be the $(m_i-1)$-plane, $M_i = M_P \cap \mathbb{P}V_{F^i_P}$. Then the set of fixed points $(M_P)^{\mathbb{G}_m} = \coprod_{i = 0}^r M_i$, which is a disjoint union. 
	\end{lem}
	\begin{proof}
		By the definition of $\mathbb{G}_m$-action on the projective space $\mathbb{P}V_{P}$, the set of fixed points is the disjoint union $\coprod_{i=0}^r\mathbb{P}V_{F_P^i}$. Hence, $(M_P)^{\mathbb{G}_m} = \coprod_{i = 0}^r M_i$.
	\end{proof}
	
	\begin{prop}\label{VMRT}
		Fix a vector space $W$. Let $P \in \operatorname{Sym}^rW^*$ be a reduced homogeneous polynomial of degree $r$. 
		If the Euler-symmetric variety $M_P\subset \mathbb{P}V_{P}$ is smooth, then for a general point $x \in M_P$, there exists a line in $\mathbb{P}V_{P}$ which lies on $M_P$. 
	\end{prop}
	\begin{proof}
		From Theorem \ref{fu20} we know $M_P$ is an equivariant compactification of the vector group $W$, it suffices to prove that at the original point $o \in M_P$, there exists a line in $\mathbb{P}V_{P}$ which lies on $M_P$. Also from Theorem \ref{fu20}, this can be reduced to prove that $V(F^2_P) \neq \emptyset$. 
		
		Note that from Proposition \ref{derivative}, $F^{r-1}_z=F^{r-1}_{P_*}=\langle dP_*\rangle$. Let $Z$ be the image of $V(P_*)$ of the following rational map:
		\[
		\Phi_{F^{r-1}_z}: \mathbb{P}T_zM_P (\supset V(P_*)) \dashrightarrow \mathbb{P}V_{F_{z}^{r-1}}(\subset \mathbb{P}V_P), [w] \mapsto [dP_*(w)].
		\]
		Since $P$ is a reduced homogeneous polynomial, $P_*$ is also a reduced homogeneous polynomial. Clearly $Z$ is the dual variety of $V(P_*)$, which is non-empty. From the proof of Proposition \ref{symmetricFund}, $Z$ is fixed by $\mathbb{G}_m$ and hence $Z\subset M_1$.

		Write the coordinates of $\mathbb{P}V_{P}$ as $[z_0:z_1:\cdots:z_m:w^{(2)}_1:\cdots:w^{(2)}_{m_2}:\cdots:w^{(r)}_1]$, and let $F^2_P = \langle Q^{(2)}_1, \cdots, Q^{(2)}_{m_2} \rangle$ generated by $Q^{(2)}_1, \cdots, Q^{(2)}_{m_2}$. Then the following homogeneous polynomials of degree two lie in the vanishing ideal of $M_P \subset \mathbb{P}V_{P}$:
		\begin{align*}
			f_1^{(2)} &= z_0w_1^{(2)} - Q_1^{(2)}(z_1, \cdots, z_m);\\
			& \vdots \\
			f_{m_2}^{(2)} &= z_0w_{m_2}^{(2)} - Q_{m_2}^{(2)}(z_1, \cdots, z_m).
		\end{align*}
		Therefore, $Z\subset M_1 \subset V(F^2_P) $. In particular, $V(F^2_P) \neq \emptyset$. 
	\end{proof}
	\begin{rem}\label{remVMRT}
		\begin{itemize}
			\item The existence of a line is equivalent to $\operatorname{ord}(\textbf{F}_P) = 1$ where the order is defined to be the largest natural number $k$ such that $V(F_P^k)=\emptyset$;
			\item Furthermore, we have $M_1 = V(F^2_P)$. Since every point in $V(F^2_P)$ corresponds to a line through $o$ on $M_P$, whose terminal point lies in $M_1$.
		\end{itemize}
	\end{rem}
	
	Next we recall the Bia{\l}ynicki-Birula decomposition theorem (over $\mathbb{C}$) as follows.
	\begin{thm}[\cite{BB:1973}, Theorem 4.1, 4.2, 4.3]\label{BBdecom}
		Let $M$ be a complete smooth complex manifold with an algebraic $\mathbb{G}_m$-action. Let $M^{\mathbb{G}_m}$ be the set of fixed points under the $\mathbb{G}_m$-action and $M^{\mathbb{G}_m}=\bigsqcup_{i\in \mathbf{I}} Y_i$ be the decomposition of $M^{\mathbb{G}_m}$ into connected components. Then there exists a unique locally closed $\mathbb{G}_m$-invariant decomposition of $M$,\[M=\bigsqcup_{i\in \mathbf{I}} M^+(Y_i) = \bigsqcup_{i\in \mathbf{I}} M^-(Y_i)\]
		and morphisms $\gamma^{\pm}_i: M^{\pm}(Y_i)\rightarrow Y_i$, $i\in \mathbf{I}$, such that the following holds.
		\begin{enumerate}
			\item $M^{\pm}(Y_i)$ is a smooth $\mathbb{C}^*$-invariant complex submanifold of $M$. $Y_i$ is a closed complex submanifold of $M^{\pm}(Y_i)$. $M^+(Y_i)\cap M^-(Y_i)=Y_i$.
			
			\item $\gamma^{\pm}_i$ is algebraic and is a $\mathbb{C}^*$-fibration over $Y_i$ (i.e., each fiber is a $\mathbb{C}^*$-module) such that $\gamma^{\pm}_i|_{Y_i}$ is the identity, for $\forall i \in \mathbf{I}$.
			
			\item Let $T_x(M)^+, T_x(M)^0, T_x(M)^-$ be the weight spaces of the isotropy action on the tangent space $T_x(M)$ with positive, zero, negative weights respectively. Then for any $x\in Y_i$, the tangent space $T_x(M^{\pm}(Y_i))=T_x(M)^0\oplus T_x(M)^{\pm}$ and the dimension of the fibration given by $\gamma_i^{\pm}$ equals $\dim T_x(M)^{\pm}$.
		\end{enumerate}
	\end{thm}
	
	We know the Bia{\l}ynicki-Birula decomposition of $\mathbb{P}V_P$ with respect to the $\mathbb{C}^*$-action given by $\mathbb{G}_m\subset GL(V_P)$ in Lemma \ref{multiAction} can be written as follows:
	\begin{align*}
		(\mathbb{P}V_P)^{\mathbb{G}_m}&=\bigsqcup^r_{k=0} \mathbb{P}V_{F^k_P},\\
		(\mathbb{P}V_P)_i^+&=\mathbb{P}(\oplus^r_{k=i}V_{F^k_P})\backslash \mathbb{P}(\oplus^r_{k=i+1}V_{F^k_P}),\\
		(\mathbb{P}V_P)_i^-&=\mathbb{P}(\oplus^i_{k=0}V_{F^k_P})\backslash \mathbb{P}(\oplus^{i-1}_{k=0}V_{F^k_P}).
	\end{align*}
	Then we let $M_i^+=(\mathbb{P}V_P)_i^+\cap M_P$ and $M_i^-=(\mathbb{P}V_P)_i^-\cap M_P$ respectively. From Iversen's fixed point theorem \cite[Proposition 1.3]{iversen1972fixed} and the smoothness of $M_P$, we know that $(M_P)^{\mathbb{G}_m}=\bigcup^r_{i=1} M_i$ is a smooth subvariety. Hence $M_i$ are also smooth for any $0\leq i \leq r$.
	\begin{prop}\label{dualdef}
		Fix a vector space $W$. Let $P \in \operatorname{Sym}^rW^*$ be a reduced homogeneous polynomial of degree $r$. 
		If the Euler-symmetric variety $M_P$ is smooth, then $M_1=Z$ and moreover $M_1$ 
		is not dual defective, i.e., its dual variety is a hypersurface. 
	\end{prop}
	\begin{proof}
		By Proposition \ref{derivative}, we know that $Z\subset \mathbb{P}V_{F_P^2}$ is the dual variety of $V(P_*) \subset \mathbb{P}T_zM_P$. 
		Therefore, from the proof of Proposition \ref{VMRT} and Remark \ref{remVMRT}, we only need to show that $M_1 = Z$.
		
		Let $L_z^r$ be the union of all $\mathbb{G}_m$-stable rational curves through $z$ whose tangent direction at $z$ lies in $\operatorname{Sm}(V(P_*))=V(P_*)\cap(dP_*\neq 0)\subset \mathbb{P}T_zM_P$. 
		Then we have \[\operatorname{dim}L_z^r = \operatorname{dim}(V(P_*)) +1 = n-1.\]
		From the definition of $L_z^r$, we know that $L_z^r\backslash \{z\} \subset M_1^+$, hence $Z$ is a union of some irreducible components of $M_1$.
		
		Since $M_1$ is smooth, the irreducible components of $M_1$ do not intersect. Let $M_1 = Z \cup Z_1$. Suppose that $Z_1 \neq \emptyset$, then $\overline{Z_1^+}$ is also a divisor. By the definition of $Z$ (see the proof of Proposition \ref{VMRT}), $Z$ is actually the set of source points of the $\mathbb{G}_m$-stable rational curves on $M_P$ through $z$ with degree $r-1$. Since $Z_1\cap Z = \emptyset$, we know $z \notin \overline{Z_1^+}$. Write the coordinates of $\mathbb{P}V_{P}$ as $[z_0:z_1:\cdots:z_m:w^{(2)}_1:\cdots:w^{(2)}_{m_2}:\cdots:w^{(r)}_1]$, then $\overline{Z_1^+}$ is contained in the hyperplane $\{w_1^{(r)}=0\}$. Also we know $\overline{Z_1^+} \subset \{z_0=0\}$.
		This implies that $\dim(\overline{Z^+_1}) \leq m-2$, which contradicts with $\operatorname{dim}(\overline{Z_1^+})=m-1$. Therefore, $M_1 = Z$.
	\end{proof}
	\begin{prop}\label{DivisorTerminal}
		Let the boundary divisor $D = \cup_{j\in J}D_j$ and $D_j$ be irreducible components of $D$, for all $j\in J$. Then the terminal point $z \in D_j$, for any $j\in J$.
	\end{prop}
	\begin{proof}
		From \cite[Proposition 2.3, Theorem 2.5]{hassett1999geometry}, we know that the action $W$ on boundary divisor $D$ stabilizes all its irreducible components. Let $D_1$ be an irreducible component of $D$, 
		and let $p = [0:\cdots:0:f^j:\cdots:f^r]$ be a point of $D_1$ where each $f^i (j\leq i\leq r)$ is viewed as an element in $(F^i_P)^*$ and $f^j \neq 0$. Then $j\geq 1$.
		
		For $v\in W$, define 
		\[\mathcal{L}_v^p = \overline{\{g_{tv}\cdot p\mid t\in \mathbb{C}\}}.
		\]
		Therefore, $\mathcal{L}_v^p$ is a rational curve and $\mathcal{L}_v^p \subset D_1$ for any vector $v\in W$.
		
		From the explicit expression of the action of $W$ on $\mathbb{P}V_{P}$ (see Remark \ref{explicitaction}), we know that
		\[
		g_{tv}\cdot p= [0:\cdots:0:g_{tv}^p\cdot f^{j}:\cdots:g_{tv}^p\cdot f^r],
		\]
		where for each $j\leq i\leq r$,
		\[
		g_{tv}^p\cdot f^i = \sum_{\ell=j}^iC_i^{\ell} f^{\ell}\circ \iota_{tv}^{i-\ell} = \sum_{\ell=j}^it^{i-\ell}C_i^{\ell} f^{\ell}\circ \iota_{v}^{i-\ell} .
		\]
		Since $F^j_P$ consists of all $(r-j)$-th derivatives of $P$ and $j\geq 1$, for general $v\in W$, $f^j\circ \iota_{v}^{r-j} \neq 0$.
		Hence $g_{tv}^p\cdot f^r\neq 0$ and the rational curve $\mathcal{L}_v^p$ contains the terminal point $z$, then we have $z \in \mathcal{L}_v^p \subset D_1$.
	\end{proof}
	\begin{rem}
		\begin{itemize}
			\item For a non-reduced polynomial $P$, this  proposition also holds.
			\item For general homogeneous polynomial $P$, the smoothness of the Euler-symmetric variety $M_P$ only depends on the point $z$. That is to say, the point $z$ is 'the most singular' point of $M_P$, and if the point $z$ is smooth then $M_P$ is smooth variety.
		\end{itemize}
	\end{rem}

	\section{Proof of Theorem \ref{mainthm}}
	Note that Theorem \ref{cartan} is only applicable for general points, then \textit{a prior} we do not know if $\textbf{F}_z$ is also a symbol system. \cite[Theorem 3.3]{ein2023vanishing} is essential to obtain that $\textbf{F}_z$ is actually a symbol system. For the reader's convenience we restate \cite[Theorem 3.3]{ein2023vanishing} as follows.
	
	Let $X\subset \mathbb{P}^N$ be a projective variety and $L=\mathcal{O}_X(1)$ be the restriction of hyperplane line bundle of $\mathbb{P}^N$ on $X$. We denote $J^p(L)$ the sheaf of $p$-jets of $L$. We have a natural short exact sequence
	\[0\rightarrow \operatorname{Sym}^{p}T^*X\otimes L\rightarrow J^{p}(L)\stackrel{\pi_{p-1}^{p}}\rightarrow J^{p-1}(L)\rightarrow 0,\]
	where $T^*X$ is the cotangent sheaf of $X$ and $\pi_{p-1}^p$ is the truncation map. Let \[\alpha_i: H^0(\mathbb{P}^N, \mathcal{O}_{\mathbb{P}^N}(1))\otimes \mathcal{O}_{X}\rightarrow J^i(L)\]
	be the Taylor series map by taking terms with order $\leq i$ in the Taylor series of a global section of $L$.
	Define $\mathbf{U}_k$ to be the maximal open subset of $X$ contained in the smooth locus of $X$ such that the quotient sheaf $J^i(L)/\operatorname{Im}(\alpha_{i})$ is locally free of constant rank over $\mathbf{U}_k$ for all $i\leq k$. And the kernel of the restriction of the truncation map $\pi_{p-1}^p: \operatorname{Im}(\alpha_p)\to \operatorname{Im}(\alpha_{p-1})$ is the $p$-th fundamental form. \cite[Theorem 3.3]{ein2023vanishing} can be interpreted as
	\begin{thm}\label{symbolsys}
		The system of fundamental forms $\mathbf{F}_x^{\leq k}:= \oplus_{i=0}^k F_x^i$ is a symbol system for any $x\in \mathbf{U}_k$.
	\end{thm}
	From this we obtain
	\begin{lem}
		Fix a vector space $W$. Let $P \in \operatorname{Sym}^rW^*$ be a reduced homogeneous polynomial of degree $r$. If the Euler-symmetric variety $M_P$ is smooth, then the system of fundamental forms $\mathbf{F}_z$ is a symbol system.
	\end{lem}
	\begin{proof}
		From above theorem, we only need to prove that the terminal point $z \in \mathbf{U}_r$. By upper semi-continuity of the rank function of a coherent sheaf (cf.\cite[Chapter III 12.7.2]{hartshorne1979algebraic}), for any point $q \in M_P\backslash \mathbf{U}_r$, there exists $2\leq i\leq r$ such that
		\[
		\operatorname{rank}_q(J^i(L)/\operatorname{Im}(\alpha_{i})) > \operatorname{rank}_x(J^i(L)/\operatorname{Im}(\alpha_{i})),
		\]
		for $x \in \mathbf{U}_r$.
		Suppose that $z \notin \mathbf{U}_r$, then $z \in \mathbf{U}_j\backslash \mathbf{U}_{j+1}$ for some $1\leq j\leq r-1$. Therefore, we have 
		\begin{align*}
			\operatorname{rank}_z(J^{j+1}(L)/\operatorname{Im}(\alpha_{j+1}))& > \operatorname{rank}_x(J^{j+1}(L)/\operatorname{Im}(\alpha_{j+1})); \\ 
			\operatorname{rank}_z(J^k(L)/\operatorname{Im}(\alpha_{k})) &= \operatorname{rank}_x(J^k(L)/\operatorname{Im}(\alpha_{k})), \forall\; k \leq j ,
		\end{align*}
		for $x \in \mathbf{U}_{j}$. Since we have the following exact sequence at point $z$
		\[
		0\to F^{j+1}_z \to \operatorname{Im}(\alpha_{j+1})_z\to \operatorname{Im}(\alpha_{j})_z \to 0,
		\]
		we have $\operatorname{dim}(F^{j+1}_{z}) < \operatorname{dim}(F^{j+1}_{x})$, which contradicts with Proposition \ref{symmetricFund}.
	\end{proof}
	Then we immediately have
	\begin{prop}\label{eulerhomoge}
		Fix a vector space $W$. Let $P \in \operatorname{Sym}^rW^*$ be a reduced homogeneous polynomial of degree $r$. If the Euler-symmetric variety $M_P$ is smooth, then $M_P$ is homogeneous.
	\end{prop}
	\begin{proof}
		Consider the rational map $\Psi_{\textbf{F}_z}$ defined by the system of fundamental forms of $M_P$ at $z$:
		\[
		\Psi_{\textbf{F}_z}: \mathbb{P}(\mathbb{C}\oplus W^*)\dashrightarrow \mathbb{P}((F_z^r)^*\oplus \cdots \oplus (F_z^2)^*\oplus W^* \oplus \mathbb{C}),
		\]
		$$
		[t:w] \mapsto [\iota^r_w : t\iota^{r-1}_w : \cdots: t^{r-2}\iota^2_w: t^{r-1}w: t^r].
		$$
		In terms of the natural identification of $\mathbb{P}((F_z^r)^*\oplus \cdots \oplus (F_z^2)^*\oplus W^* \oplus \mathbb{C})$ and $\mathbb{P}(V_{P})$, $M_P = \operatorname{Im}(\Psi_{\textbf{F}_z})$. Since $\textbf{F}_z$ is a symbol system, the $W^*$-action on $\mathbb{P}(\mathbb{C}\oplus W^*)$ can be extended to $M_P$, making the orbit of $z$ is biregular to $W^*$. Let $U_z$ denote the open orbit of $z$. 
		By Proposition \ref{DivisorTerminal}, we know that for any point $p\in D$ in the boundary divisor, there exists a vector $v\in W$ such that $g_v \cdot p \in U_z$. 
		
		Since $U_z \cap U \subset M_P$ is also an open subset of $M_P$, there exists a vector $w \in W^*$ such that $g_w \cdot (g_v \cdot p) \in U$. Hence $M_P$ is homogeneous under the algebraic group generated by $W$ and $W^*$ in $GL(V_P)$.
	\end{proof}
	
	\begin{prop}\label{PicVMRT}
		Fix a vector space $W$. Let $P \in \operatorname{Sym}^rW^*$ be a reduced homogeneous polynomial of degree $r$. 
		If the Euler-symmetric variety $M_P$ is smooth, then the Picard number $\rho(M_P)$ equals the number of irreducible components of $M_1$, namely,
		\[
		\rho(M_P) = \#_{irr}M_1,
		\]
		where $\#_{irr}Y$ is the number of irreducible components of $Y$.
	\end{prop}
	\begin{proof}
		From the proof of Proposition \ref{dualdef} and \cite[Theorem 2.5]{hassett1999geometry}, we have 
		\[
		\rho(M_P) \geq \#_{irr} M_1.
		\]
		If $\rho(M_P) \neq \#_{irr} M_1$, then $\overline{M_1^+} \subsetneq D$, and there exist an irreducible component $D_1$ of $D$ and an irreducible component $K$ of $M_j$, for some $2\leq j\leq r-1$, such that $D_1 = \overline{K^+}$.
		
		Let $L_z^k$ be the closure of the union of all $\mathbb{G}_{m}$-stable rational curves through $z$ such that the tangent direction at $z$ lies in $V(F_z^k)\subset \mathbb{P}T_zM_P$. Therefore, $D_1 \subset L_z^{r-j+1}$ by Proposition \ref{DivisorTerminal}. But we have 
		\[
		\operatorname{dim}(L_z^{r-j+1}) \leq \operatorname{dim}(V(F_z^{r-j+1}))+1\leq \operatorname{dim}(V(F_z^{r-1}))+1\leq n-2,
		\]
		which contradicts with the fact that $D_1$ is a divisor.
	\end{proof}
	\begin{rem}\label{CounterPicVMRT}
		Consider the negative Bia{\l}ynicki-Birula decomposition for the $\mathbb{G}_m$-action $M_P = \coprod_{i=0}^rM_i^-$.  
		we have the similar results as follows:
		\begin{enumerate}
			\item $M_r =  z$, the terminal point;
			\item $M_r^-$ is biregular to an affine space of dimension $n$, and the boundary divisor $D_z = \coprod_{i=0}^{r-1}M_i^- = \overline{M^-_{r-1}}$;
			\item $M_{r-1}\subset \mathbb{P}V_{F^{r-1}}$ is also the dual variety of $V(P)\subset \mathbb{P}T_oM_P$;
			\item all the irreducible component of $D_z$ contain the original point $o$;
			\item $\rho(M_P) = \#_{irr}M_{r-1}$.
		\end{enumerate}  
	\end{rem}
	\begin{cor}
		Assumption as above, then we have
		\[
		\rho(M_P) = \#_{irr}V(P) = \#_{irr}V(P_*).
		\]
	\end{cor}
	\begin{proof}
		Since $M_0^+$ and $M_r^-$ are affine open subvariety which is biregular to affine space, then
		\[
		\#_{irr} D = \rho(M_P) = \#_{irr}D_z.
		\]
		From Theorem \ref{PicVMRT} and Remark \ref{CounterPicVMRT}, we know that $L_o^r = D_z$, $L_z^r = D$. From definition, $L_o^r$ is also the closure of the $\mathbb{G}_m$-fibration $\widetilde{L_o^r}$:
		\[
		\xymatrix{
			\widetilde{L_o^r} \ar@{^{(}->}[r] \ar[d] & M_0^+\backslash o \ar[d] \\ 
			\operatorname{Sm}(V(P)) \ar@{^{(}->}[r]& \mathbb{P}T_oM_P.
		}
		\]
		Therefore, $\#_{irr}L^r_o = \#_{irr}\widetilde{L_o^r} =  \#_{irr}\operatorname{Sm}(V(P)) =  \#_{irr}V(P)$.
		Similarly, $ \#_{irr}L^r_z =  \#_{irr}V(P_*)$. 
	\end{proof}
	
	Now we can finish the proof of Theorem \ref{mainthm}.
	\begin{proof}
		From Proposition \ref{eulerhomoge}, $M_P$ is homogeneous. Since Euler-symmetric varieties are rational, 
		hence $M_P$ is rational homogeneous and it is actually a Hermitian symmetric space as it is an equivariant compactification of a vector group (see \cite{Arz2011equiva}). As $P$ is reduced, $V(P)$ is reduced. $M_P$ is a product of irreducible compact Hermitian symmetric spaces, then $M_1$ is the disjoint union of the VMRTs of each factors. Since $M_1$ is not dual defect, we know each factor of $M_P$ must be isomorphic to an irreducible compact Hermitian symmetric spaces listed in Theorem \ref{mainthm}. The embedding is obvious.
		
		For the converse part, it can be deduced directly from the irreducible case, which is well-known (see for example \cite{landsberg2002construction}). 
	\end{proof}
	\begin{rem}
		Even if the polynomial $P$ is non-reduced, the Proposition \ref{eulerhomoge} still holds. Consequently, we can relax the requirement that $P$ is reduced in Theorem \ref{mainthm}. However, this only allows us to conclude that $M_P$ can be factored as a product of those  irreducible Hermitian symmetric spaces listed in Theorem \ref{mainthm}, without specifying the exact embedding.
	\end{rem}
	
	
	\section*{Acknowledgement}
	The authors would like to thank Jun-Muk Hwang and Baohua Fu for some useful suggestions and comments. The second author would also like to thank Wenbo Niu for introducing his work and helpful discussions. The first author was supported by a start-up funding of Shenzhen University (000001032064). The second author was supported by the Institute for Basic Science (IBS-R032-D1-2024-a00).
\bibliographystyle{alpha}
\bibliography{ref}
\end{document}